\documentclass[12pt]{article}
\usepackage{amsmath,amsthm, amssymb}
\usepackage{amssymb,latexsym}

\newtheorem{theorem}{Theorem}

\newtheorem{lemma}{Lemma}

\begin{document}

\baselineskip=17pt

\title{\bf The Bombieri -- Vinogradov theorem for primes of the form $\mathbf{p=x^2+y^2+1}$}

\author{\bf S. I. Dimitrov}

\date{2024}

\maketitle
\begin{abstract}
In this paper, we establish a Bombieri -- Vinogradov type result for prime numbers of the form $p=x^2+y^2+1$. The proof is based on the enveloping sieve.
\quad\\
{\bf  2020 Math.\ Subject Classification}: 11B25 $\cdot$ 11L20 $\cdot$ 11L40 $\cdot$ 11N05 $\cdot$ 11N36
\end{abstract}

%\dedicatory{Dedicated to}

\section{Introduction and statement of the result}
\indent

The Bombieri -- Vinogradov theorem is extremely important result in analytic number theory and has various applications.
It concerns the distribution of primes in arithmetic progressions, averaged over a range of moduli. It asserts that when $X\geq2$ and $B>0$ is a fixed, then
\begin{equation}\label{Bomb-Vin}
\sum\limits_{d\le \sqrt{X}/(\log X)^{B+5}}\max\limits_{(a,\,d)=1}\Bigg|\sum_{p\le X\atop{p\equiv a\,( \textmd{mod}\, d)}}1-\frac{\pi(X)}{\varphi(d)}\Bigg|\ll\frac{X}{\log^BX}\,,
\end{equation}
where $\pi(X)$ is the prime-counting function and $\varphi (n)$ is Euler's function.
An interesting problem in number theory is the proof of the Bombieri -- Vinogradov theorem for prime numbers of a special form.
In this connection Peneva \cite{Peneva}, Wang and Cai \cite{Cai}, Lu \cite{Lu} and J. Li, M. Zhang and F. Xue \cite{Zhang-Li} proved a Bombieri -- Vinogradov type result
for Piatetski-Shapiro \cite{Shapiro} primes $p=[n^c]$. 
Recently Nath \cite{Nath} obtained a Bombieri-Vinogradov type theorems for Maynard \cite{Maynard} primes which are primes with a missing digit. 
On the other hand in 1960 Linnik \cite{Linnik} showed that there exist infinitely many prime numbers of the form
$p=x^2 + y^2 +1$, where $x$ and $y$ are integers.
More precisely he proved the formula
\begin{equation}\label{Linnnik}
\sum_{p\leq X}r(p-1)=\pi\prod_{p>2}\bigg(1+\frac{\chi(p)}{p(p-1)}\bigg)\frac{X}{\log X}+\mathcal{O}\Bigg(\frac{X(\log\log X)^5}{(\log X)^{1+\theta_0}}\Bigg)\,,
\end{equation}
where $r(n)$ is the number of solutions of the equation $n=x^2 + y^2$ in integers, $\chi(n)$ is the non-principal character modulo 4 and
\begin{equation}\label{theta0}
\theta_0=\frac{1}{2}-\frac{1}{4}e\log2=0.0289...
\end{equation}
A sharper estimate for the remainder term in \eqref{Linnnik} was established by y Bredikhin \cite{Bredikhin}.
Motivated by these studies and applying the enveloping sieve we establish a Bombieri -- Vinogradov type result for primes of the form $p=x^2+y^2+1$ that are called Linnik primes.
More precisely we prove the following theorem.
\begin{theorem}\label{Theorem}
Let $X>e$ and $A>0$ is a fixed. Then
\begin{equation*}
\sum\limits_{q\le \log^AX\atop{(q, a)=1}}\Bigg|\sum\limits_{p\leq X\atop{p\equiv a\, (\textmd{mod}\, q)}}r(p-1)-\frac{1}{\varphi(q)}\sum\limits_{p\leq X}r(p-1)\Bigg|
\ll\frac{X(\log\log X)^7}{(\log X)^{1+\theta_0}}\,.
\end{equation*}
\end{theorem}
As we can see, the range of moduli $q\leq\log^AX$ in the theorem is a lot shorter than the one in the classical Bombieri-Vinogradov theorem.
This is due to the fact that in addition to the $X^{1/2}$-barrier are added the restrictions of the enveloping sieve.

\section{Notations and outline of the proof}
\indent

Assume that $X\geq2$ and $A>0$. The letter $p$ will always denote prime number.
As usual $\tau _k(n)$ is the number of solutions of the equation $m_1m_2\ldots m_k$ $=n$ in natural numbers $m_1,\,\ldots,m_k$.
We denote by $(d,q)$, $[d,q]$ the greatest common divisor and the least common multiple of $d$ and $q$ respectively.
The function $\Omega(n)$ counts the total number of prime factors of $n$ honoring their multiplicity.
Instead of $m\equiv n\,\pmod {k}$ we write for simplicity $m\equiv n(k)$. Denote
\begin{align}
\label{Q}
&Q=\log^AX\,;\\
\label{D}
&D=\frac{X^{\frac{1}{2}}}{(\log X)^{A+14}}\,;\\
\label{Y}
&Y=X^{1/(\log\log X)^2}\,;\\
\label{P}
&P=\prod_{p\leq Y} p\,;\\
\label{gn}
&g(n)=
\begin{cases}1\,,\;\; \mbox{ if n is a prime not exceeding Y },\\
0\,,\;\; \mbox{ else }; 
\end{cases}
\end{align}
\begin{equation}\label{hn}
\hspace{-35mm} h(n)=\begin{cases}1\,,\;\; \mbox{ if } (n, P)=1\,,\\
0\,,\;\; \mbox{ else };
\end{cases}
\end{equation}
\begin{equation}\label{fn}
\hspace{-54mm} f(n)=g(n)+h(n)\,.
\end{equation}
The function $f(n)$ is called the function of the enveloping sieve. It is clear that $f(n)=1$ when $n$ is a prime.
When the system 
\begin{equation*}
\left|\begin{array}{ll}
p \equiv a_1 \, (d)  \\
p \equiv a_2 \, (q)
\end{array}
\right.
\end{equation*}
has a solution we denote it with $l(d,q)$. From Chinese remainder theorem we have that if  $(a_1, d)=(a_2, q)=1$ then $p\equiv l(d,q)\,([d,q])$ and $(l(d,q), [d,q])=1$.

In the first step of proof by \eqref{Q}, \eqref{D} and the well-known identity 
\begin{equation*}
r(n)=4\sum_{d|n}\chi(d)=4\left( \sum\limits_{d|n\atop{d\leq D}}\chi(d)+\sum\limits_{d|n\atop{D<d<X/D}}\chi(d)+\sum\limits_{d|n\atop{d\geq X/D}}\chi(d)\right)
\end{equation*}
we write
\begin{equation}\label{decomposition}
\sum\limits_{q\le \log^AX\atop{(q, a)=1}}\Bigg|\sum\limits_{p\leq X\atop{p\equiv a\, ( q )}}r(p-1)-\frac{1}{\varphi(q)}\sum\limits_{p\leq X}r(p-1)\Bigg|\ll S_1(X)+S_2(X)+S_3(X)+S_4(X)\,,
\end{equation}
where
\begin{align}
\label{S1}
&S_1(X)=\sum\limits_{q\le Q\atop{(q, a)=1}}\Bigg|\sum\limits_{p\leq X\atop{p\equiv a\, ( q )}}\sum\limits_{d|p-1\atop{d\leq D}}\chi(d)-\frac{1}{\varphi(q)}\sum\limits_{p\leq X}\sum\limits_{d|p-1\atop{d\leq D}}\chi(d)\Bigg|\,,\\
\label{S2}
&S_2(X)=\sum\limits_{q\le Q\atop{(q, a)=1}}\Bigg|\sum\limits_{p\leq X\atop{p\equiv a\, ( q )}}\sum\limits_{d|p-1\atop{d\geq X/D}}\chi(d)-\frac{1}{\varphi(q)}\sum\limits_{p\leq X}\sum\limits_{d|p-1\atop{d\geq X/D}}\chi(d)\Bigg|\,,\\
\label{S3}
&S_3(X)=\sum\limits_{q\le Q\atop{(q, a)=1}}\frac{1}{\varphi(q)}\sum\limits_{p\leq X}\Bigg|\sum\limits_{d|p-1\atop{D<d<X/D}}\chi(d)\Bigg|\,,\\
\label{S4}
&S_4(X)=\sum\limits_{q\le Q\atop{(q, a)=1}}\Bigg|\sum\limits_{p\leq X\atop{p\equiv a\, ( q )}}\sum\limits_{d|p-1\atop{D<d<X/D}}\chi(d)\Bigg|\,.
\end{align}
We shall estimate $S_1(X)$, $S_2(X)$, $S_3(X)$ and $S_4(X)$, respectively, in the sections \ref{SectionS1}, \ref{SectionS2}, \ref{SectionS3}   and \ref{SectionS4}.
In section \ref{Sectionfinal} we shall finalize the proof of Theorem \ref{Theorem}.
The range of moduli in $S_1(X)$ corresponds to the Bombieri-Vinogradov theorem, for $S_2(X)$ we will apply the same method as for $S_1(X)$,  
for $S_3(X)$ we are going to use the fact that the sum is rather short, thus we can get some saving through Lemma \ref{Hooley1}.
Our main difficulties will be in estimating the sum $S_4(X)$. It is there we will apply the enveloping sieve. Note that sum $S_4(X)$ gives the largest contribution.

\section{Preliminary lemmas}
\indent

\begin{lemma}\label{Hooley1}
Let $\omega>0$. Then
\begin{equation*}
\sum\limits_{p\leq X}
\bigg|\sum\limits_{d|p-1\atop{\sqrt{X}(\log X)^{-\omega}<d<\sqrt{X}(\log X)^{\omega}}}\chi(d)\bigg|\ll \frac{X(\log\log X)^5}{(\log X)^{1+\theta_0}}\,.
\end{equation*}
\end{lemma}
\begin{proof}
The proof of very similar result (with $\omega$ = 48 and with the condition $d \, |\, N-p$ rather than $d \,| \, p-1$) is available in (\cite{Hooley}, Chapter 5). 
The reader will easily see that the methods used there yield also the validity of Lemma \ref{Hooley1}.
\end{proof}

\begin{lemma}\label{Brun-Titchmarsh} (Brun -- Titchmarsh theorem)
Let $q<X$ and $(q, a)=1$. Then
\begin{equation*}
\sum\limits_{p\leq X\atop{p\equiv a\, (q)}}1<\frac{2X}{\varphi(q)\log\frac{ 2X}{q}}\,.
\end{equation*}
\end{lemma}
\begin{proof}
See (\cite{Montgomery-Vaughan}, Theorem 2).
\end{proof}
Using Brun -- Titchmarsh theorem one can prove the following lemma.
\begin{lemma}\label{Hooley2}
Let $r<n/2$ and $\mathcal{N}(n,r)$ is the number of solutions of the equation
\begin{equation*} 
p_1+rp_2=n 
\end{equation*}
in primes $p_1,p_2$. Then 
\begin{equation*}
\mathcal{N}(n,r)\ll\frac{n^2}{\varphi(nr)\log^2\frac{n}{r}}\,.
\end{equation*}
\end{lemma}
\begin{proof}
See (\cite{Hooley}, Lemma 2).
\end{proof}

\begin{lemma}\label{Hooley11}
Let $y\leq X$, $0<\delta<1$, $k=\mathcal{O}\big(X^\delta\big)$ and $a=\prod\limits_{p_i} p^{a_i}_i$. Then for the function $f(n)$  defined by  \eqref{fn} we have
\begin{equation*}
\sum\limits_{n<y\atop{n\equiv a\, (k)}}f(n)
=\begin{cases}\frac{1}{\varphi(k)}B(X)y+\mathcal{O}\left(\frac{X}{k\log^5X}\right),\;\; \mbox{ if } \;\; (a^{(1)}, k)=1\,,\\
\hspace{24mm}\mathcal{O}\left(\frac{X}{k\log^5X}\right),\;\; \mbox{ if } \;\; (a^{(1)}, k)>1\,,
\end{cases}
\end{equation*}
where
\begin{equation*}
B(X)\ll\frac{(\log\log X)^2}{\log X}\,,
\end{equation*}
\begin{equation*}
a^{(1)}=\prod_{p_i\leq Y} p^{a_i}_i\,,
\end{equation*}
where the implied constants are absolute.
\end{lemma}
\begin{proof}
See (\cite{Hooley}, Lemma 11).
\end{proof}

\begin{lemma}\label{Hooley12}
Let $\frac{1}{2}\leq a\leq\frac{7}{4}$. Then
\begin{equation*}
\sum\limits_{n\leq y} a^{\Omega(n)} \ll y (\log2y)^{a-1}  \,.
\end{equation*}
\end{lemma}
\begin{proof}
See (\cite{Hooley}, Lemma 12). For more general result we refer to (\cite{Tenenbaum}, Chapter II.6, Theorem 2).

\end{proof}

\begin{lemma}\label{Hooley13}
Let $\frac{1}{2}\leq \alpha<1$, $\omega>0$ and $y>e^e$. Then 
\begin{equation*}
\sum\limits_{\sqrt{y}(\log y)^{-\omega}<n<\sqrt{y}(\log y)^\omega\atop{\Omega(n)\leq \alpha\log\log y}}\frac{1}{n}\ll (\log y)^{\gamma_\alpha-1}\log\log y\,,
\end{equation*}
where $\gamma_\alpha=\alpha-\alpha\log \alpha$.
\end{lemma}
\begin{proof}
The proof of very similar result (with $\omega$ = 48) is available in (\cite{Hooley}, Lemma 13). 
Inspecting the arguments presented in (\cite{Hooley}, Lemma 13), the reader will readily see that the proof of Lemma \ref{Hooley13} can be obtained is the same manner.
\end{proof}

\begin{lemma}\label{Hooley14}
Let $(rs, n)=1$ and $r,s,n\leq X$. Then 
\begin{align*}
\sum\limits_{l\geq y\atop{(l, ns)=1}}\frac{\chi(l)}{\varphi(rsl)}=\mathcal{O}\Big((\log\log X)R_n(r,s,y)\Big)&+\mathcal{O}\Bigg((\log\log X)\frac{\sigma_{-1}(s)}{rs}\sigma_{-1}(n,y)\Bigg)\\
&+\mathcal{O}\Bigg(\frac{(\log\log X)^2}{rsy}\Bigg)\,,
\end{align*}
where 
\begin{align}
\label{sigman}
&\sigma_{-1}(n)=\sum\limits_{d|n}\frac{1}{d}\,;\\
\label{sigmany}
&\sigma_{-1}(n,y)=\sum\limits_{d|n\atop{d>y}}\frac{1}{d}\,;
\end{align}

\begin{align}
\label{Rnrsy}
&R_n(r,s,y)=\frac{\log 2y}{y}\frac{\tau_2(s)}{rs}\tau(n,y)\,;\\
\label{tauny}
&\tau(n,y)=\sum\limits_{d|n\atop{d\leq y}}1\,,
\end{align}
where the implied constants are absolute.
\end{lemma}
\begin{proof}
See (\cite{Hooley}, Lemma 14).
\end{proof}

\begin{lemma}\label{Hooley15}
Let $1<u<X$, $u'\geq u$, $\omega>0$ and $n\leq X$. Then 
\begin{align*}
&\sum\limits_{h\leq u}\sum\limits_{u/h<d<u(\log X)^\omega /h} R_n(h,d,u'/h) \ll (\log\log X)^4\,,\\
&\sum\limits_{h\leq u}\sum\limits_{u/h<d<u(\log X)^\omega /h} \frac{\sigma_{-1}(d)}{hd}\sigma_{-1}(n,u'/h) \ll (\log\log X)^3\,,\\
&\sum\limits_{h\leq u}\sum\limits_{u/h<d<u(\log X)^\omega /h} \frac{h}{u} \frac{1}{hd}\ll \log\log X\,.
\end{align*}
\end{lemma}
\begin{proof}
The proof of very similar result (with $\omega$ = 96) is available in (\cite{Hooley}, Lemma 15).
The reader will easily see that the methods used there yield also the validity of Lemma \ref{Hooley15}.
\end{proof}

\begin{lemma}\label{Murty}
Let $X\geq2$. Then
\begin{equation*}
\sum\limits_{n\leq X}\frac{1}{\varphi(n)}\ll \log X\,.
\end{equation*}
\end{lemma}
\begin{proof}
See (\cite{Murty}, Ch .4, Ex. 4.4.14).
\end{proof}

\section{Estimation of $\mathbf{S_1(X)}$}\label{SectionS1}
\indent

Using \eqref{Bomb-Vin}, \eqref{Q}, \eqref{D}, \eqref{S1}, the trivial bound $|\chi(d)|\leq1$, the triangle inequality, Chinese remainder theorem, 
the elementary formula $[d,q]=d \, \frac{q}{(d,q)}\,, \Big(d, \, \frac{q}{(d,q)}\Big)=1$ and Lemma \ref{Murty} we obtain 
\begin{align*}
S_1(X)&=\sum\limits_{q\le Q\atop{(q, a)=1}}\Bigg|\sum\limits_{d\leq D}\chi(d)\Bigg(\sum\limits_{p\leq X\atop{p\equiv 1\, (d)\atop{p\equiv a\, (q)}}}1-\frac{1}{\varphi(q)}\sum\limits_{p\leq X\atop{p\equiv 1\, (d)}}1\Bigg)\Bigg|\nonumber\\
&\ll\sum\limits_{q\le Q}\sum\limits_{d\leq D}\Bigg|\sum\limits_{p\leq X\atop{p\equiv l(d,q)\, ([d,q])\atop{(l(d,q), [d,q])=1}}}1-\frac{1}{\varphi(q)}\sum\limits_{p\leq X\atop{p\equiv 1\, (d)}}1\Bigg|\nonumber\\
\end{align*}

\begin{align}\label{S1est1}
&\ll\sum\limits_{q\le Q}\sum\limits_{d\leq D\atop{(d, q)=1}}\Bigg|\sum\limits_{p\leq X\atop{p\equiv l(d,q)\, (dq)\atop{(l(d,q), dq)=1}}}1-\frac{1}{\varphi(q)}\sum\limits_{p\leq X\atop{p\equiv 1\, (d)}}1\Bigg|\nonumber\\
&\ll\sum\limits_{q\le Q}\sum\limits_{d\leq D\atop{(d, q)=1}}\Bigg|\frac{1}{\varphi(q)}\sum\limits_{p\leq X\atop{p\equiv 1\, (d)}}1-\frac{\pi(X)}{\varphi(d)\varphi(q)}\Bigg|
+\sum\limits_{q\le Q}\sum\limits_{d\leq D\atop{(d, q)=1}}\Bigg|\sum\limits_{p\leq X\atop{p\equiv l(d,q)\, (dq)\atop{(l(d,q), dq)=1}}}1-\frac{\pi(X)}{\varphi(d)\varphi(q)}\Bigg|\nonumber\\
&\ll\sum\limits_{q\le Q}\frac{1}{\varphi(q)}\sum\limits_{d\leq D}\Bigg|\sum\limits_{p\leq X\atop{p\equiv 1\, (d)}}1-\frac{\pi(X)}{\varphi(d)}\Bigg|+S'_1
\ll\frac{X}{(\log X)^{A+9}}\log Q+S'_1\nonumber\\
&\ll\frac{X\log\log X}{(\log X)^{A+9}}+S'_1(X)\,,
\end{align}
where
\begin{equation}\label{S'1}
S'_1(X)=\sum\limits_{q\le Q}\sum\limits_{d\leq D\atop{(d, q)=1}}\Bigg|\sum\limits_{p\leq X\atop{p\equiv l(d,q)\, (dq)\atop{(l(d,q), dq)=1}}}1-\frac{\pi(X)}{\varphi(d)\varphi(q)}\Bigg|\,.
\end{equation}
Bearing in mind \eqref{Bomb-Vin}, \eqref{Q}, \eqref{D}, \eqref{S'1}, Cauchy -- Schwarz inequality and the upper bounds 
\begin{equation*}
\sum\limits_{n\le X}\frac{\tau_2^2(n)}{n}\ll \log^4X\,,
\end{equation*}
\begin{equation*}
\Bigg|\sum\limits_{p\leq X\atop{p\equiv l(h)\, (h)\atop{(l(h), h)=1}}}1-\frac{\pi(X)}{\varphi(h)}\Bigg|\ll\frac{X}{h}
\end{equation*}
we get
\begin{align}\label{S'1est}
S'_1(X)&\ll\sum\limits_{h\le QD}\tau_2(h)\Bigg|\sum\limits_{p\leq X\atop{p\equiv l(h)\, (h)\atop{(l(h), h)=1}}}1-\frac{\pi(X)}{\varphi(h)}\Bigg|\nonumber\\
&\ll\left(\sum\limits_{h\le QD}\tau_2^2(h)\right)^{\frac{1}{2}}\left(\sum\limits_{h\le QD}\Bigg|\sum\limits_{p\leq X\atop{p\equiv l(h)\, (h)\atop{(l(h), h)=1}}}1-\frac{\pi(X)}{\varphi(h)}\Bigg|^2\right)^{\frac{1}{2}}\nonumber\\
&\ll X^{\frac{1}{2}}\left(\sum\limits_{h\le QD}\frac{\tau_2^2(h)}{h}\right)^{\frac{1}{2}}
\left(\sum\limits_{h\le QD}\Bigg|\sum\limits_{p\leq X\atop{p\equiv l(h)\, (h)\atop{(l(h), h)=1}}}1-\frac{\pi(X)}{\varphi(h)}\Bigg|\right)^{\frac{1}{2}}\ll\frac{X}{\log^2X}\,.
\end{align}
Now \eqref{S1est1} and \eqref{S'1est} imply
\begin{equation}\label{S1est2}
S_1(X)\ll\frac{X}{\log^2X}\,.
\end{equation}

\section{Estimation of $\mathbf{S_2(X)}$}\label{SectionS2}
\indent

Since
\begin{equation}\label{pm1chij}
\sum\limits_{d|p-1\atop{d\geq X/D}}\chi(d)=\sum\limits_{m|p-1\atop{m\leq (p-1)D/X}}\chi\bigg(\frac{p-1}{m}\bigg)
=\sum\limits_{j=\pm1}\chi(j)\sum\limits_{m|p-1\atop{m\leq (p-1)D/X\atop{\frac{p-1}{m}\equiv j\,(4)}}}1
\end{equation}
Using \eqref{S2}, \eqref{pm1chij} and working as for $S_1(X)$ we deduce
\begin{align}\label{S2est}
S_2(X)&=\sum\limits_{q\le Q\atop{(q, a)=1}}\Bigg|\sum\limits_{m<D\atop{2|m}}\sum\limits_{j=\pm1}\chi(j)\Bigg(\sum\limits_{mX/D+1\leq p\leq X\atop{p\equiv 1+jm\, (4m)\atop{p\equiv a\, (q)}}}1-\frac{1}{\varphi(q)}
\sum\limits_{mX/D+1\leq p\leq X\atop{p\equiv 1+jm\, (4m)}}1\Bigg)\Bigg|\nonumber\\
&\ll\sum\limits_{q\le Q\atop{(q, a)=1}}\sum\limits_{m<D\atop{2|m}}\Bigg|\sum\limits_{mX/D+1\leq p\leq X\atop{p\equiv 1+jm\, (4m)\atop{p\equiv a\, (q)}}}1-\frac{1}{\varphi(q)}
\sum\limits_{mX/D+1\leq p\leq X\atop{p\equiv 1+jm\, (4m)}}1\Bigg|\nonumber\\
&\ll\frac{X}{\log^2X}\,.
\end{align}

\section{Estimation of $\mathbf{S_3(X)}$}\label{SectionS3}
\indent

Using \eqref{Q}, \eqref{D}, \eqref{S3}, Lemma \ref{Hooley1} and Lemma \ref{Murty} we find
\begin{align}\label{S3est}
S_3(X)\ll\frac{X(\log\log X)^5}{(\log X)^{1+\theta_0}}\sum\limits_{q\le Q}\frac{1}{\varphi(q)}
\ll\frac{X(\log\log X)^6}{(\log X)^{1+\theta_0}}\,.
\end{align}

\section{Estimation of $\mathbf{S_4(X)}$}\label{SectionS4}
\indent

In this section our argument is a modification of Hooley's \cite{Hooley} argument.
\begin{lemma}\label{Hooley13q}
Let $1<\alpha\leq\frac{3}{2}$ and $y>e^e$. Then 
\begin{equation*}
\sum\limits_{n\leq y\atop{\Omega(n)>\alpha\log\log y-1\atop{n\equiv 0\, (q)}}}\frac{1}{n}\ll \frac{\alpha^{\Omega(q)}}{q}(\log y)^{\gamma_\alpha}\log\log y\,,
\end{equation*}
where $\gamma_\alpha=\alpha-\alpha\log \alpha$.
\end{lemma}
\begin{proof}
By Abel's summation formula and Lemma \ref{Hooley12} it follows
\begin{align*}
\sum\limits_{n\leq y\atop{\Omega(n)>\alpha\log\log y-1\atop{n\equiv 0\, (q)}}}\frac{1}{n}&\ll\alpha^{1-\alpha\log\log y}\sum\limits_{n\leq y\atop{n\equiv 0\, (q)}}\frac{\alpha^{\Omega(n)}}{n}
=\frac{\alpha^{\Omega(q)+1-\alpha\log\log y}}{q}\sum\limits_{m\leq y/q}\frac{\alpha^{\Omega(m)}}{m}\nonumber\\
&\ll\frac{\alpha^{\Omega(q)+1-\alpha\log\log y}}{q}(\log y)\max_{1\leq t \leq y/q}\left(\frac{1}{t}\sum\limits_{m\leq t}\alpha^{\Omega(m)}\right)\\
&\ll\frac{\alpha^{\Omega(q)+1-\alpha\log\log y}}{q}(\log y)\max_{1\leq t \leq y/q}\left(\frac{1}{t}t(\log y)^{\gamma_\alpha-1}\right)\\
&\ll \frac{\alpha^{\Omega(q)}}{q}(\log y)^{\gamma_\alpha}\log\log y
\end{align*}
which proves the lemma.
\end{proof}

The next lemma is an analog of Lemma \ref{Hooley1} over arithmetic progressions.
\begin{lemma}\label{Enveloping}
Let $(q,a)=1$. Then
\begin{equation*}
\sum\limits_{p\leq X\atop{p\equiv a\, (q)}}\sum\limits_{d|p-1\atop{D<d<X/D}}\chi(d)\ll\frac{X(\log\log X)^5}{(\log X)^{1+\theta_0}}\frac{(e/2)^{\frac{\Omega(q)}{2}}}{(q\varphi(q))^{\frac{1}{2}}}\,.
\end{equation*}
\end{lemma}
\begin{proof}
For any fixed $p$ and $q$ we define
\begin{align}
\label{Epq}
&E(p,q)=\sum\limits_{D<d<X/D\atop{p\equiv l(d,q)\,(dq)\atop{(l(d,q), dq)=1\atop{(d,q)=1}}}}1\,,\\
\label{Fpq}
&F(p,q)=\sum\limits_{D<d<X/D\atop{p\equiv l(d,q)\,(dq)\atop{(l(d,q), dq)=1\atop{(d,q)=1}}}}\chi(d)\,.
\end{align}
By \eqref{S4}, \eqref{Epq}, \eqref{Fpq}, Chinese remainder theorem and Cauchy -- Schwarz inequality we derive
\begin{align}\label{S4est1}
\sum\limits_{p\leq X\atop{p\equiv a\, (q)}}\sum\limits_{d|p-1\atop{D<d<X/D}}\chi(d)&=\sum\limits_{D<d<X/D}\chi(d)\sum\limits_{p\leq X\atop{p\equiv 1\, (d)\atop{p\equiv a\, (q)}}}1
=\sum\limits_{D<d<X/D}\chi(d)\sum\limits_{p\leq X\atop{p\equiv l(d,q)\, ([d,q])\atop{(l(d,q), [d,q])=1}}}1\nonumber\\
&\ll\Bigg|\sum\limits_{p\leq X}\sum\limits_{D<d<X/D\atop{p\equiv l(d,q)\,(dq)\atop{(l(d,q), dq)=1\atop{(d,q)=1}}}}\chi(d)\Bigg|
=\Bigg|\sum\limits_{p\leq X}F(p,q)\Bigg|=\Bigg|\sum\limits_{p\leq X\atop{E(p,q)\neq0}}F(p,q)\Bigg|\nonumber\\
&\ll\left(\Sigma_E\right)^{\frac{1}{2}}\left(\Sigma_F\right)^{\frac{1}{2}}\,,
\end{align}
where
\begin{align}
\label{SigmaE}
&\Sigma_E=\sum\limits_{p\leq X\atop{E(p,q)\neq0}}1\,,\\
\label{SigmaF}
&\Sigma_F=\sum\limits_{p\leq X}F^2(p,q)\,.
\end{align}

\textbf{Upper bound of} $\mathbf{\Sigma_E}$

Assume that $1<\alpha<\frac{3}{2}$. From \eqref{SigmaE} we have
\begin{equation}\label{SigmaEest1}
\Sigma_E\leq\sum\limits_{p\leq X}E_1(p,q)+\sum\limits_{p\leq X\atop{E_2(p,q)\neq0}}1=\Sigma^{(1)}_E+\Sigma^{(2)}_E\,,
\end{equation}
where
\begin{align}
\label{SigmaE1}
&\Sigma^{(1)}_E=\sum\limits_{p\leq X}E_1(p,q)\,,\\
\label{SigmaE2}
&\Sigma^{(2)}_E=\sum\limits_{p\leq X\atop{E_2(p,q)\neq0}}1\,,\\
\label{E1pq}
&E_1(p,q)=\sum\limits_{D<d<X/D\atop{p\equiv l(d,q)\,(dq)\atop{(l(d,q), dq)=1\atop{(d,q)=1\atop{\Omega(p-l(d,q))\leq\alpha\log\log X}}}}}1\,,\\
\label{E2pq}
&E_2(p,q)=\sum\limits_{D<d<X/D\atop{p\equiv l(d,q)\,(dq)\atop{(l(d,q), dq)=1\atop{(d,q)=1\atop{\Omega(p-l(d,q))>\alpha\log\log X}}}}}1\,.
\end{align}
We first estimate the sum $\Sigma^{(1)}_E$. We denote the conditions 
\begin{align*}
&(D) \;\; :  \;\; D<d<X/D\,,\\
&(M) \;\; :  \;\; D/q\log^2X<m<X/D\,,\\
&(P) \;\; : \;\; p\leq X\,, \;\;  p-l(d,q)=dqm\,, \;\;  (l(d,q), dq)=1 \,, \;\; (d,q)=1\,.
\end{align*}
When the condition $(P)$ is fulfilled then the condition for summation in $\Sigma^{(1)}_E$ means that one of the two quantities  $\Omega(dq)$, $\Omega(m)$ is not greater than $\frac{1}{2}\alpha\log\log X$. 
Taking into account this consideration and using consistently \eqref{SigmaE1}, \eqref{E1pq}, Lemma \ref{Brun-Titchmarsh}, the inequality
\begin{equation}\label{varphiest}
\frac{n}{\varphi(n)}\ll\log\log n
\end{equation}
and Lemma \ref{Hooley13} we obtain

\begin{align}\label{SigmaE1est}
\Sigma^{(1)}_E&\leq\sum\limits_{(D), (P)\atop{\Omega(dq)\leq\frac{1}{2}\alpha\log\log X}}1+\sum\limits_{(D), (P)\atop{\Omega(m)\leq\frac{1}{2}\alpha\log\log X}}1\nonumber\\
&\leq\sum\limits_{(D), (P)\atop{\Omega(d)\leq\frac{1}{2}\alpha\log\log X}}1+\sum\limits_{(M), (P)\atop{\Omega(m)\leq\frac{1}{2}\alpha\log\log X}}1+\sum\limits_{d\leq X/D\atop{m\leq D/q\log^2X}}1\nonumber\\
&\ll\sum\limits_{(M), (P)\atop{\Omega(m)\leq\frac{1}{2}\alpha\log\log X}}1+\frac{X}{q\log^2X}\ll
\sum\limits_{(M)\atop{\Omega(m)\leq\frac{1}{2}\alpha\log\log X}}\sum\limits_{p\leq X\atop{p\equiv l(d,q)\, (mq)\atop{(l(d,q),mq)=1}}}1+\frac{X}{q\log^2X}\nonumber\\
&\ll\frac{X}{\log X}\sum\limits_{(M)\atop{\Omega(m)\leq\frac{1}{2}\alpha\log\log X}}\frac{1}{\varphi(mq)}+\frac{X}{q\log^2X}\nonumber\\
&\ll\frac{X\log\log X}{q\log X}\sum\limits_{(M)\atop{\Omega(m)\leq\frac{1}{2}\alpha\log\log X}}\frac{1}{m}+\frac{X}{q\log^2X}\ll\frac{X(\log\log X)^2}{q(\log X)^{2-\frac{\alpha}{2}+\frac{\alpha}{2}\log\frac{\alpha}{2}}}\,.
\end{align}
Next we estimate the sum $\Sigma^{(2)}_E$. From \eqref{SigmaE2} and \eqref{E2pq} we get 
\begin{equation}\label{SigmaE2est1}
\Sigma^{(2)}_E\leq\sum\limits_{n\leq X\atop{n\equiv 0\,(q)\atop{\Omega(n)>12\log\log X}}}1+\sum\limits_{p\leq X\atop{E_3(p,q)\neq0}}1\,,
\end{equation}
where
\begin{equation}\label{E3pq}
E_3(p,q)=\sum\limits_{D<d<X/D\atop{p\equiv l(d,q)\,(dq)\atop{(l(d,q), dq)=1\atop{(d,q)=1\atop{\alpha\log\log X<\Omega(p-l(d,q))\leq12\log\log X}}}}}1\,.
\end{equation}
Now Lemma \ref{Hooley12} gives us
\begin{align}\label{Omega>12est}
\sum\limits_{n\leq X\atop{n\equiv 0\,(q)\atop{\Omega(n)>12\log\log X}}}1&\ll\frac{1}{\log^3X}\sum\limits_{m\leq X/q}(\sqrt[4]{e})^{\Omega(mq)}
=\frac{(\sqrt[4]{e})^{\Omega(q)}}{\log^3X}\sum\limits_{m\leq X/q}(\sqrt[4]{e})^{\Omega(m)}\nonumber\\
&\ll\frac{(\sqrt[4]{e})^{\Omega(q)}}{q}\frac{X}{(\log X)^{4-\sqrt[4]{e}}}\ll\frac{(\sqrt[4]{e})^{\Omega(q)}}{q}\frac{X}{\log^2X}\,.
\end{align}
Further we denote by $R_x$ the set of natural numbers $n\leq X$ that have no prime factor greater than $X^{\frac{1}{20\log\log X}}$.
We also define

\begin{align}
\label{E4pq}
&E_4(p_1,q)=\sum\limits_{D<d<X/D\atop{p_1-l(d,q)=rp_2\atop{r\equiv 0\, (q)\atop{(l(d,q), dq)=1\atop{(d,q)=1\atop{p_1-l(d,q)\notin R_x\atop{\Omega(p_1-l(d,q))>\alpha\log\log X}}}}}}}1\,,\\
\label{E5pq}
&E_5(p_1,q)=\sum\limits_{D<d<X/D\atop{p_1-l(d,q)=rp_2\atop{r\equiv 0\, (q)\atop{(l(d,q), dq)=1\atop{(d,q)=1}}}}}1\,.
\end{align}
Now \eqref{varphiest}, \eqref{E3pq}, \eqref{E4pq}, \eqref{E5pq}, Lemma \ref{Hooley2} and Lemma \ref{Hooley13q} yield 
\begin{align}\label{E3pqest}
\sum\limits_{p\leq X\atop{E_3(p,q)\neq0}}1&\leq\sum\limits_{n\in R_x\atop{\Omega(n)\leq12\log\log X}}1+\sum\limits_{p_1\leq X\atop{E_4(p_1,q)\neq0}}1
\ll X^{\frac{3}{5}}+\sum\limits_{r<X^{1-{1/20\log\log X}}\atop{\Omega(r)>\alpha\log\log X-1\atop{r\equiv 0\, (q)}}}\sum\limits_{p_1\leq X\atop{E_5(p_1,q)\neq0}}1\nonumber\\
&\ll X^{\frac{3}{5}}+\frac{X(\log\log X)^2}{\log^2X}\sum\limits_{r\leq X\atop{\Omega(r)>\alpha\log\log X-1\atop{r\equiv 0\, (q)}}}\frac{1}{r}\nonumber\\
&\ll\frac{\alpha^{\Omega(q)}}{q}\frac{X(\log\log X)^3}{(\log X)^{2-\alpha+\alpha\log\alpha}}\,.
\end{align}
Combining \eqref{SigmaE2est1}, \eqref{Omega>12est} and \eqref{E3pqest} we deduce
\begin{equation}\label{SigmaE2est2}
\Sigma^{(2)}_E\ll\frac{(\sqrt[4]{e})^{\Omega(q)}}{q}\frac{X}{\log^2X}+\frac{\alpha^{\Omega(q)}}{q}\frac{X(\log\log X)^3}{(\log X)^{2-\alpha+\alpha\log\alpha}}\,.
\end{equation}
Bearing in mind \eqref{SigmaEest1}, \eqref{SigmaE1est}, \eqref{SigmaE2est2} and choosing $\alpha=e/2$  we derive
\begin{equation}\label{SigmaEest2}
\Sigma_E\ll\frac{(e/2)^{\Omega(q)}}{q}\frac{X(\log\log X)^3}{(\log X)^{2-\frac{e}{2}\log2}}\,.
\end{equation}
It remains to estimate the sum $\Sigma_F$.

\textbf{Upper bound of} $\mathbf{\Sigma_F}$

Using \eqref{gn}, \eqref{hn}, \eqref{fn}, \eqref{Fpq}, \eqref{SigmaF} and Chinese remainder theorem  we write
\begin{align*}
\Sigma_F=\sum\limits_{p\leq X}F^2(p,q)\leq\sum\limits_{n\leq X}F^2(n,q)f(n)
=\sum\limits_{n\leq X}\sum\limits_{D<d'_1, d'_2<X/D\atop{n\equiv l(d'_1, q)\,(d'_1q)\atop{n\equiv l(d'_2, q)\,(d'_2q)\atop{(l(d'_1, q), d'_1q)=1\atop{(l(d'_2, q), d'_2q)=1\atop{(d'_1, q)=(d'_2, q)=1}}}}}}\chi(d'_1)\chi(d'_2)f(n)
\end{align*}

\begin{align}\label{SigmaFest1}
=\sum\limits_{n\leq X}\sum\limits_{D<d'_1, d'_2<X/D\atop{n\equiv l(d'_1, d'_2, q)\,([d'_1q, d'_2q])\atop{(l(d'_1, d'_2, q), [d'_1q, d'_2q])=1\atop{(d'_1d'_2, q)=1}}}}\chi(d'_1)\chi(d'_2)f(n)\,.
\end{align}
Put $(d'_1, d'_2)=h$, $d'_1=d_1h$, $d'_2=d_2h$, $l(d'_1, d'_2, q)=l(h, q)$. Then $[d'_1q, d'_2q]=d_1d_2hq$ and $(d_1, d_2)=(d_1d_2h, q)=(d_1d_2hq, l(h, q))=1$.
We denote the conditions 
\begin{align*}
&(D_i) \;\; :  \;\; D/h<d_i<X/hD\,, \quad i=1, \,2 \,,\\
&(D_3) \;\; :  \;\; (d_1, d_2)=1\,,\\
&(K) \;\; :  \;\; (d_1d_2hq, l(h, q))=(d_1d_2h, q)=1\,,\\
&(K_1) \;\; : \;\;  (d_1hq, l(h, q))=1\,.
\end{align*}
Thus \eqref{SigmaFest1} leads to
\begin{align}\label{SigmaFest2}
\Sigma_F\leq\sum\limits_{n\leq X\atop{n\equiv l(h, q)\,(d_1d_2hq)\atop{(D_1), (D_2), (D_3), (K)}}}\chi^2(h)\chi(d_1)\chi(d_2)f(n)=\Sigma^{(1)}_F+\Sigma^{(2)}_F\,,
\end{align}
where
\begin{align}
\label{SigmaF1}
&\Sigma^{(1)}_F=\sum\limits_{n\leq X\atop{h\geq X^{1/8}\atop{n\equiv l(h, q)\,(d_1d_2hq)\atop{(D_1), (D_2), (D_3), (K)}}}}\chi^2(h)\chi(d_1)\chi(d_2)f(n)\,,\\
\label{SigmaF2}
&\Sigma^{(2)}_F=\sum\limits_{n\leq X\atop{h<X^{1/8}\atop{n\equiv l(h, q)\,(d_1d_2hq)\atop{(D_1), (D_2), (D_3), (K)}}}}\chi^2(h)\chi(d_1)\chi(d_2)f(n)\,.\\
\end{align}
We first estimate the sum $\Sigma^{(1)}_F$.
The condition for summation in $\Sigma^{(1)}_F$ means that $d_1d_2hq<X^{7/8}(\log X)^{3A+28}$. 
Taking into account this consideration and using \eqref{SigmaF1} and Lemma \ref{Hooley11} we obtain
\begin{align*}
\Sigma^{(1)}_F&=\sum\limits_{h\geq X^{1/8}\atop{(D_1), (D_2), (D_3), (K)}}\chi^2(h)\chi(d_1)\chi(d_2)\sum\limits_{n\leq X\atop{n\equiv l(h, q)\,(d_1d_2hq)}}f(n)\nonumber\\
&=XB(X)\sum\limits_{h\geq X^{1/8}\atop{(D_1), (D_2), (D_3), (K)}}\frac{\chi^2(h)\chi(d_1)\chi(d_2)}{\varphi(d_1d_2hq)}
+\mathcal{O}\Bigg(\frac{X}{\log^5X}\sum\limits_{h\geq X^{1/8}\atop{(D_1), (D_2), (D_3), (K)}}\frac{1}{d_1d_2hq}\Bigg)\nonumber\\
\end{align*}

\begin{align}\label{SigmaF1est1}
&=\frac{XB(X)}{\varphi(q)}\sum\limits_{X^{1/8}\leq h\leq D \atop{(D_1), (D_2), (D_3), (K)}}\frac{\chi^2(h)\chi(d_1)\chi(d_2)}{\varphi(d_1d_2h)}\nonumber\\
&+\frac{XB(X)}{\varphi(q)}\sum\limits_{D< h<X/D\atop{(D_1), (D_2), (D_3), (K)}}\frac{\chi^2(h)\chi(d_1)\chi(d_2)}{\varphi(d_1d_2h)}
+\mathcal{O}\Bigg(\frac{X}{q\log^5X}\sum\limits_{d_1, d_2, h\leq X}\frac{1}{d_1d_2h}\Bigg)\nonumber\\
&=\frac{XB(X)}{\varphi(q)}\Sigma_1+\frac{XB(X)}{\varphi(q)}\Sigma_2+\mathcal{O}\Bigg(\frac{X}{q\log^2X}\Bigg)\,,
\end{align}
where
\begin{align}
\label{Sigma1}
&\Sigma_1=\sum\limits_{X^{1/8}\leq h\leq D\atop{(D_1), (D_2), (D_3), (K)}}\frac{\chi^2(h)\chi(d_1)\chi(d_2)}{\varphi(d_1d_2h)}\,,\\
\label{Sigma2}
&\Sigma_2=\sum\limits_{D< h<X/D\atop{(D_1), (D_2), (D_3), (K)}}\frac{\chi^2(h)\chi(d_1)\chi(d_2)}{\varphi(d_1d_2h)}\,.
\end{align}
Now \eqref{sigman}, \eqref{sigmany}, \eqref{Rnrsy}, \eqref{Sigma1}, Lemma \ref{Hooley14} and Lemma \ref{Hooley15} imply
\begin{align}\label{Sigma1est}
&\Sigma_1=\sum\limits_{(D_1)\atop{X^{1/8}\leq h\leq D}}\chi^2(h)\chi(d_1) \sum\limits_{ (D_2), (K_1)\atop{(d_2, d_1l(h, q))=1}}\frac{\chi(d_2)}{\varphi(d_1d_2h)}\nonumber\\
&\ll\sum\limits_{(D_1)\atop{ h\leq D}}\Bigg[(\log\log X)\Bigg(R_{l(h, q)}\bigg(h, d_1, \frac{D}{h}\bigg)+R_{l(h, q)}\bigg(h, d_1, \frac{X}{hD}\bigg)\Bigg)\nonumber\\
&+(\log\log X)\frac{\sigma_{-1}(d_1)}{hd_1}\Bigg(\sigma_{-1}\left(l(h, q), \frac{D}{h}\right)+\sigma_{-1}\left(l(h, q),  \frac{X}{hD}\right)\Bigg)+\frac{h}{D}\frac{(\log\log X)^2}{hd_1}\Bigg]\nonumber\\
&\ll (\log\log X)^5\,.
\end{align}
On the other hand by \eqref{varphiest} and \eqref{Sigma2} it follows
\begin{equation}\label{Sigma2est}
\Sigma_2\ll(\log\log X)\sum\limits_{D< h<X/D\atop{d_1, d_2<X/D^2}}\frac{1}{hd_1d_2}\ll (\log\log X)^4\,.
\end{equation}
Bearing in mind \eqref{SigmaF1est1}, \eqref{Sigma1est}, \eqref{Sigma2est} and and Lemma \ref{Hooley11}  we get
\begin{equation}\label{SigmaF1est2}
\Sigma^{(1)}_F\ll\frac{1}{\varphi(q)}\frac{X(\log\log X)^7}{\log X}\,.
\end{equation}
Next we estimate the sum $\Sigma^{(2)}_F$. We denote the conditions 
\begin{align*}
&(R) \;\; :  \;\; D/ht<r<X/htD\,,\\
&(S) \;\; :  \;\; D/ht<s<X/htD\,,\\
&(H) \;\; :  \;\; h<X^{1/8}\,,\\
&(HT) : \;\;  h<X^{1/8}\,, \quad  t<X^{1/8}\,, \\
&(K_2) \;\; :  \;\; (rsthq, l(h, q))=(rsth, q)=1\,,\\
&(K_3) \;\; : \;\;  (rthq, l(h, q))=(rth, q)=1\,.
\end{align*} 
From \eqref{SigmaF2} and the formula
\begin{equation*}
\sum\limits_{rt=d_1\atop{st=d_2}}\mu(t)=
\begin{cases}
1\,,\;\; \mbox{ if } \;\; (d_1, d_2)=1\,,\\
0\,,\;\; \mbox{ if } \;\; (d_1, d_2)>1
\end{cases}
\end{equation*}
we deduce
\begin{equation}\label{SigmaF2est1}
\Sigma^{(2)}_F=\sum\limits_{n\leq X\atop{rst^2hqm=n-l(h, q)\atop{(R), (S), (H), (K_2)}}}\mu(t)\chi^2(t)\chi^2(h)\chi(r)\chi(s)f(n)=\Sigma_3+\Sigma_4\,,
\end{equation}
where
\begin{align}
\label{Sigma3}
&\Sigma_3=\sum\limits_{n\leq X\atop{t< X^{1/8}\atop{rst^2hqm=n-l(h, q)\atop{(R), (S), (H), (K_2)}}}}\mu(t)\chi^2(t)\chi^2(h)\chi(r)\chi(s)f(n)\,,\\
\label{Sigma4}
&\Sigma_4=\sum\limits_{n\leq X\atop{t\geq X^{1/8}\atop{rst^2hqm=n-l(h, q)\atop{(R), (S), (H), (K_2)}}}}\mu(t)\chi^2(t)\chi^2(h)\chi(r)\chi(s)f(n)\,.
\end{align}
The condition for summation in $\Sigma_3$ means that 
\begin{equation}\label{rst2hqm}
rt^2hqm=\frac{n-l(h, q)}{s}<\frac{Xht}{D}<X^{\frac{3}{4}}(\log X)^{A+14}\,.
\end{equation}
Put
\begin{align}
\label{y1}
&y_1=rtqmD+l(h, q)\,,\\
\label{y2}
&y_2=\max\Big(X, rtqmXD^{-1}+l(h, q)\Big)\,,\\
\label{lambda}
&\lambda=rt^2hqm\,.
\end{align}
We recall that
\begin{equation}\label{chi4}
\chi(s)=
\begin{cases}
\;\;\, 1\,,\;\; \mbox{ if } \;\; s\equiv 1\,(4)\,,\\
-1\,,\;\; \mbox{ if } \;\; s\equiv -1\,(4)\,,\\
\;\;\, 0\,,\;\; \mbox{ else }\,. 
\end{cases}
\end{equation}
Now \eqref{Sigma3}, \eqref{rst2hqm} -- \eqref{chi4} and Lemma \ref{Hooley11} give us
\begin{align}\label{Sigma3est}
\Sigma_3&=\sum\limits_{rt^2hqm<X^{3/4}(\log X)^{A+14}\atop{(R), (HT)}}\mu(t)\chi^2(t)\chi^2(h)\chi(r)
\sum\limits_{n\leq X\atop{n=s\lambda+l(h, q)\atop{(S), (K_2)}}}\chi(s)f(n)\nonumber\\
&\ll\sum\limits_{rt^2hqm<X^{3/4}(\log X)^{A+14}\atop{(R), (HT)}}
\Bigg|\sum\limits_{y_1<n\leq y_2\atop{n=s\lambda+l(h, q)\atop{(K_2)}}}\chi(s)f(n)\Bigg|\nonumber\\
&=\sum\limits_{rt^2hqm<X^{3/4}(\log X)^{A+14}\atop{(R), (HT)}}
\Bigg|\sum\limits_{y_1<n\leq y_2\atop{n\equiv l(h, q)+\lambda\,(4\lambda)\atop{(K_3)}}}f(n)-\sum\limits_{y_1<n\leq y_2\atop{n\equiv l(h, q) -\lambda\,(4\lambda)\atop{(K_3)}}}f(n)\Bigg|\nonumber\\
&=\sum\limits_{rt^2hqm<X^{3/4}(\log X)^{A+14}\atop{(R), (HT)}}
\Bigg|\frac{y_2-y_1}{\varphi(4\lambda)}B(X)-\frac{y_2-y_1}{\varphi(4\lambda)}B(X)+\mathcal{O}\Bigg(\frac{X}{\lambda\log^5X}\Bigg) \Bigg|\nonumber\\
&\ll \frac{X}{q\log^5X}\sum\limits_{r, t, h, m\leq X}\frac{1}{rt^2hm}\ll \frac{X}{q\log^2X}\,,
\end{align}
where the equivalence of the conditions $(l(h, q)+\lambda, 4\lambda^{(1)})=1$ and  $(l(h, q)-\lambda, 4\lambda^{(1)})=1$ is used.
By \eqref{Sigma4} and the upper bound
\begin{equation*}
\sum\limits_{n\le X}\tau_4(n)\ll X\log^3X\,,
\end{equation*}
we derive
\begin{align}\label{Sigma4est}
\Sigma_4=\sum\limits_{t\geq X^{1/8}\atop{rst^2hqm\leq X}}1\ll\sum\limits_{t\geq X^{1/8}}\sum\limits_{u\leq X/t^2q}\tau_4(u)\ll\frac{X\log^3X}{q}\sum\limits_{t\geq X^{1/8}}\frac{1}{t^2}\ll\frac{X^{\frac{7}{8}}\log^3X}{q}\,.
\end{align}
Now \eqref{SigmaF2est1}, \eqref{Sigma3est} and \eqref{Sigma4est} yield
\begin{equation*}
\Sigma^{(2)}_F\ll \frac{X}{q\log^2X}
\end{equation*}
which together with \eqref{SigmaFest2} and \eqref{SigmaF1est2} leads to
\begin{equation}\label{SigmaFest3}
\Sigma_F\ll\frac{1}{\varphi(q)}\frac{X(\log\log X)^7}{\log X}\,.
\end{equation}
Bearing in mind \eqref{theta0}, \eqref{S4est1}, \eqref{SigmaEest2} and \eqref{SigmaFest3} we complete the proof of the lemma. 
\end{proof}
We are now in a position to estimate the sum $S_4$. Using consistently \eqref{S4}, Lemma \ref{Enveloping}, \eqref{varphiest}, Abel's summation formula and Lemma \ref{Hooley12} we get
\begin{equation}\label{S4est}
S_4(X)\ll\frac{X(\log\log X)^7}{(\log X)^{1+\theta_0}}\,.
\end{equation}

\section{The end of the proof}\label{Sectionfinal}
\indent

Summarizing \eqref{decomposition}, \eqref{S1est2}, \eqref{S2est}, \eqref{S3est}   and \eqref{S4est} we establish Theorem \ref{Theorem}.

\vskip20pt
\footnotesize
\begin{flushleft}
S. I. Dimitrov\\
\quad\\
Faculty of Applied Mathematics and Informatics\\
Technical University of Sofia \\
Blvd. St.Kliment Ohridski 8 \\
Sofia 1756, Bulgaria\\
e-mail: sdimitrov@tu-sofia.bg\\
\end{flushleft}

\begin{flushleft}
Department of Bioinformatics and Mathematical Modelling\\
Institute of Biophysics and Biomedical Engineering\\
Bulgarian Academy of Sciences\\
Acad. G. Bonchev Str. Bl. 105, Sofia 1113, Bulgaria \\
e-mail: xyzstoyan@gmail.com\\
\end{flushleft}

\end{document}